%% file: sectional-hyperbolicity-2-inner-product11.tex
\documentclass{amsart}

\usepackage[pdftex]{graphicx} 
\usepackage[pctex32]{color}

\usepackage[all]{xy}

\chardef\bslash=`\\ 

\usepackage{graphics,graphicx}

\newcommand{\cl}{\operatorname{Cl}}

\newcommand{\Sing}{\operatorname{Sing}}

\newtheorem{theorem}{Theorem}[section] 
\newtheorem{lemma}[theorem]{Lemma}     
\newtheorem{corollary}[theorem]{Corollary}

\newtheorem{defi}[theorem]{Definition}

\newtheorem{conjecture}[theorem]{Conjecture}

\newtheorem*{hausdorff-limit}{Theorem E}

\newcommand{\eval}[2][\right]{\relax
  \ifx#1\right\relax \left.\fi#2#1\rvert}

\title[On the intersection of sectional-hyperbolic sets]{On the intersection of sectional-hyperbolic sets}
\author{S. Bautista, C.A. Morales}

\address{Instituto de Matem\'atica, Universidade Federal do Rio de Janeiro, P. O. Box 68530, 21945-970 Rio
de Janeiro, Brazil.}
\email{serafin@impa.br, morales@impa.br}

\thanks{Partially supported by CNPq, FAPERJ, CAPES, PNPD and PRONEX/DS from Brazil and UNAL from Colombia.}


\begin{document}

\begin{abstract}
We analyse the intersection of positively and negatively sectional-hyperbolic sets
for flows on compact manifolds. First
we prove that such an intersection is hyperbolic if the intersecting sets are both transitive
(this is false without such a hypothesis). Next we prove that, in general,
such an intersection consists of a nonsingular hyperbolic set,
finitely many singularities and regular orbits joining them.
Afterward we exhibit a three-dimensional star flow
with two homoclinic classes, one being positively (but not negatively)
sectional-hyperbolic and
the other negatively (but not positively) sectional-hyperbolic,
whose intersection reduces to a single periodic orbit.
This provides a counterexample to a conjecture by Shy, Zhu, Gan and Wen (\cite{sgw}, \cite{zgw}).
\end{abstract}

\maketitle
\section{Introduction} 
\label{intro}

\noindent
Dynamical systems is concerned with the study of the asymptotic behavior of the orbits of a given system.
Certain hypothesis like Smale's hyperbolicity guarantee the knowledge of this behavior.
Indeed, the celebrated {\em Smale spectral decomposition theorem}
asserts that every hyperbolic system on a compact manifold comes equipped with finite many pairwise disjoint
compact invariant sets (homoclinic classes or singularities) to which every trajectory converge.
Although present in a number of interesting examples,
such a hypothesis is far from being abundant in the dynamical forrest.
This triggered several attempts to extend it
including the {\em sectional-hyperbolicity} \cite{memo}, committed to merge
the hyperbolic theory to the so-called {\em geometric}
and {\em multidimensional Lorenz attractors} \cite{abs}, \cite{bpv}, \cite{gw}.
A number of results from the hyperbolic theory have been carried out to the sectional-hyperbolic context.
This is nowadays matter of study in a number of works, see \cite{ap} and references therein.
One of these results was
motivated by the well-known fact
that two different homoclinic classes contained in a common hyperbolic set
are disjoint.
It was quite natural to ask if this statement is also true in the sectional-hyperbolic context too.
In other words, are two different homoclinic classes contained in a common sectional-hyperbolic set
disjoint?
But recent results dealing with this question say that the answer is negative
\cite{mp2}, \cite{mp3}.
Moreover, \cite{mp3} studied the dynamics
of nontransitive sectional-Anosov flows with dense periodic orbits nowadays called {\em venice masks}.
It was proved that three-dimensional venice masks with a unique singularity exists \cite{bmp}
and that their maximal invariant set consists of two different homoclinic classes with nonempty intersection \cite{mp3}.
Venice mask with $n$ singularities can be constructed for $n\geq3$ whereas ones
with just two singularities have not been constructed yet.

These fruitful results motivate
a related problem which is the analysis of the intersection of a sectional-hyperbolic set for the flow
and a sectional-hyperbolic set for the reversed flow.
For simplicity we keep the terms {\em positively} and {\em negatively sectional-hyperbolic} for these sets
(respectively) which was coined by Shy, Gan and Wen in their recent paper \cite{sgw}.
After observing that every hyperbolic set can be realized as such an intersection, we
show an example where such an intersection is not hyperbolic.
Next we show that such an intersection is hyperbolic if the intersecting sets
are both transitive.
In general the intersection
consists of a nonsingular hyperbolic set (possibly empty),
finitely many singularities and regular orbits joining them.
Finally, we construct a three-dimensional star flow
exhibiting two homoclinic classes, one being positively
(but not negatively) sectional-hyperbolic
and the other being negatively (but not positively) sectional-hyperbolic,
whose intersection reduces to a single periodic orbit.
This will provide a counterexample to a conjecture by Zhu, Gan and Wen \cite{zgw} (as amended by Shy, Gan and Wen \cite{sgw}).

\section{Statement of the results}

\noindent
Let $M$ be a differentiable manifold endowed with a Riemannian metric $\langle\cdot,\cdot\rangle$
an induced norm $\|\cdot\|$.
We call {\em flow} any $C^1$ vector field $X$ with
induced flow $X_t$ of $M$.
If $\dim(M)=3$, then we say that $X$ is a {\em three-dimensional flow}.
We denote by $\Sing(X)$ the set of singularities (i.e. zeroes) of $X$.
By a {\em periodic point} we mean a point $x\in M$ for which there is a minimal $t>0$
satisfying $X_t(x)=x$. By an {\em orbit} we mean $O(x)=\{X_t(x):t\in\mathbb{R}\}$ and by a {\em periodic orbit}
we mean the orbit of a periodic point.
We say that $\Lambda\subset M$ is
{\em invariant} if $X_t(\Lambda)=\Lambda$ for all $t\in\mathbb{R}$.
In such a case we write $\Lambda^*=\Lambda\setminus \Sing(X)$.
We say that $\Lambda\subset M$ is {\em transitive} if there is $x\in \Lambda$ such that
$\omega(x)=\Lambda$, where $\omega(x)$ is the {\em $\omega$-limit set},
$$
\omega(x)=\left\{ y\in M:y=\lim_{n\to\infty} X_{t_n}(x)\mbox{ for some sequence }t_n\to\infty\right\}.
$$
The {\em $\alpha$-limit set} $\alpha(x)$ is the $\omega$-limit set for the reversed flow $-X$.
If the set of periodic points of $X$ in $\Lambda$ is dense in $\Lambda$, we say that $\Lambda$ has {\em dense periodic points}.

A compact invariant set $\Lambda$ is {\em hyperbolic}
if there is a continuous invariant splitting $T_\Lambda M=E^s\oplus E^X\oplus E^u$
and positive numbers $K,\lambda$ such that
\begin{enumerate}
\item $E^s$ is {\em contracting}, i.e.,
$\| DX_t(x)v^s_x\|\leq Ke^{-\lambda t}\|v^s_x\|$ for every $x\in\Lambda$, $v^s_x\in E^s_x$ and $t\geq0$.
\item $E^X_ x$ is the subspace generated by $X(x)$ in $T_x M$, for every $x\in \Lambda$.
\item $E^u$ is {\em expanding}, i.e., $\| DX_t(x)v^u_x\|\geq K^{-1}e^{\lambda t}\|v^u_x\|$ for every $x\in\Lambda$, $v^u_x\in E^u_x$ and $t\geq0$.
\end{enumerate}

A singularity or periodic orbit is hyperbolic if it does as a compact invariant set of $X$.
The elements of a (resp. hyperbolic) periodic orbit
will be called (resp. hyperbolic) periodic points.
A singularity or periodic orbit is a {\em sink} (resp. {\em source})
if its unstable subbundle $E^u$ (resp. stable subbundle $E^s$) vanishes.
Otherwise we call it {\em saddle type}.

The invariant manifold theory \cite{hps} asserts that through any point $x$ of a hyperbolic set it passes a pair of invariant manifolds,
the so-called stable and unstable manifolds
$W^{s}(x)$ and $W^{u}(x)$, tangent at $x$ to the subbundles $E^s_x$ and $E^u_x$ respectively.
Saturating them with the flow we obtain the weak stable and unstable manifolds $W^{ws}(x)$ and $W^{wu}(x)$ respectively.

On the other hand,
a compact invariant set $\Lambda$ {\em has a dominated splitting with respect to the tangent flow}
if there are an invariant splitting $T_\Lambda M=E\oplus F$
and positive numbers $K,\lambda$ such that
$$
\| DX_t(x)e_x\|\cdot \|f_x\| \leq Ke^{-\lambda t} \| DX_t(x)f_x\|\cdot\|e_x\|,
\quad\quad\forall x\in \Lambda, t\geq0, (e_x,f_x)\in E_x\times F_x.
$$

Notice that this definition allows
every compact invariant set $\Lambda$ to have a dominated splitting with respect to the tangent flow:
Just take
$E_x=T_xM$ and $F_x=0$ for every $x\in \Lambda$
(or $E_x=0$ and $F_x=T_xM$ for every $x\in\Lambda$).
However, such splittings
need not to exist under certain constraints.
For instance, not every compact invariant set has a
dominated splitting $T_\Lambda M=E\oplus F$ with respect to the tangent flow
which is {\em nontrivial}, i.e., satisfying
$E_x\neq0\neq F_x$ for every $x\in \Lambda$.

A compact invariant set $\Lambda$ is {\em partially hyperbolic}
if it has a {\em partially hyperbolic splitting},
i.e., a dominated splitting $T_\Lambda M=E\oplus F$ with respect to the tangent
flow whose dominated subbundle $E$ is contracting in the sense of (1) above.

The Riemannian metric $\langle\cdot,\cdot\rangle$ of $M$ induces a {\em $2$-Riemannian metric} \cite{mv},
$$
\langle u,v/w\rangle_p=\langle u,v\rangle_p\cdot \langle w,w\rangle_p-\langle u,w\rangle_p\cdot \langle v,w\rangle_p,
\quad\forall p\in M,\forall u,v,w\in T_pM.
$$
This in turns induces a {\em $2$-norm} \cite{g} (or {\em areal metric} \cite{ka}) defined by
$$
\|u,v\|=\sqrt{\langle u,u/v\rangle_p},
\,\,\,\,\,\,\forall p\in M,\forall u,v\in T_pM.
$$
Geometrically, $\|u,v\|$ represents the area of the paralellogram generated by $u$ and $v$ in $T_pM$.

If a compact invariant set $\Lambda$ has a dominated splitting $T_\Lambda M=E\oplus F$ with respect to the tangent flow,
then we say that its central subbundle $F$ is {\em sectionally expanding} (resp. {\em sectionally contracting}) if
$$
\| DX_t(x) u, DX_t(x) v\|\geq K^{-1}e^{\lambda t}\| u,v\|,
\quad\quad\forall x\in \Lambda, u, v\in F_x, t\geq0.
$$
(resp.
$$
\| DX_t(x) u, DX_t(x) v\|\leq Ke^{-\lambda t}\| u,v\|,
\quad\quad\forall x\in \Lambda, u, v\in F_x, t\geq0.)
$$
By a {\em sectional-hyperbolic splitting} for $X$ over $\Lambda$
we mean a partially hyperbolic splitting $T_\Lambda M=E\oplus F$
whose central subbundle $F$ is sectionally expanding.

Now we define sectional-hyperbolic set.

\begin{defi}
A compact invariant set $\Lambda$ is {\em sectional-hyperbolic} for $X$
if its singularities are hyperbolic and if there is a sectional-hyperbolic splitting for $X$ over $\Lambda$.
Following \cite{sgw}
we use the term
{\em positively} (resp. {\em negatively) sectional-hyperbolic}
to indicate a sectional-hyperbolic set for $X$ (resp. $-X$).
The corresponding sectional-hyperbolic splitting will be termed
{\em positively} (resp. {\em negatively}) {\em sectional-hyperbolic splitting}.
\end{defi}

This definition is slightly different from the original one given in Definition 2.3 of \cite{memo}
(which requires, for instance, that the central subnbundle be two-dimensional at least).
Such a difference permits {\em every} hyperbolic set $\Lambda$ to be
both positively and negatively sectional-hyperbolic.
Indeed, if $T_\Lambda M=E^s\oplus E^X\oplus E^u$ is the respective hyperbolic splitting, then
$T_\Lambda M=E^s\oplus E^{se}$ with $E^{se}=E^X\oplus E^u$
and $T_\Lambda M=\hat{E}^s\oplus \hat{E}^{se}$ with
$\hat{E}^s=E^u$
and $\hat{E}^{se}=E^s\oplus E^X$
define positively and negatively sectional-hyperbolic splittings
respectively over $\Lambda$.
In particular, {\em every hyperbolic set is the intersection of a positively and a negatively
sectional-hyperbolic set}.

\medskip

One can ask if the hyperbolic sets are the sole possible intersection between
a positively and a negatively sectional-hyperbolic set,
but they aren't.
In fact, there are nonhyperbolic compact invariant sets which, nevertheless, are
both positively and negatively sectional-hyperbolic.
This is the case of the example described in Figure \ref{fig1}. In such a figure
$O(x)$ represents the orbit of $x\in W^s(\sigma_1)\cap W^u(\sigma_2)$
whereas a singularity of a three-dimensional flow is {\em Lorenz-like} for $X$
if it has three real eigenvalues $\lambda_1,\lambda_2,\lambda_3$ satisfying
$\lambda_2<\lambda_3<0<-\lambda_3<\lambda_1$.

\begin{figure}[htv] 
\begin{center}
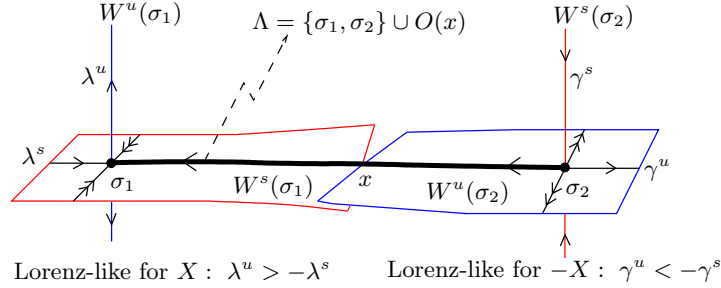 
\caption{Nonhyperbolic but positively and negatively sectional-hyperbolic.}\label{fig1}
\end{center}
\end{figure}

This counterexample motivates the search of sufficient conditions under which
the intersection of a positively and a negatively sectional-hyperbolic set be hyperbolic.
Our first result is about this problem.

\begin{theorem}
\label{th2}
The intersection of a transitive positively sectional-hyperbolic set and a transitive
negatively sectional-hyperbolic set is hyperbolic.
\end{theorem}

Consequently,

\begin{corollary}
\label{th1}
Every transitive set which is both positively and negatively sectional-hyperbolic is hyperbolic.
\end{corollary}

The similar results replacing transitivity by denseness of periodic orbits hold.

By looking at Figure \ref{fig1} we observe that this example consists of two singularities and a regular point $x$
whose $\omega$-limit and $\alpha$-limit set is a singularity.
This observation is the motivation for the result below.

\begin{theorem}
\label{th3}
Every compact invariant set which is both positively and negatively sectional-hyperbolic
is the disjoint union of
a (possibly empty) nonsingular hyperbolic set $H$,
a (possibly empty) finite set of singularities $S$
and a (possibly empty) set of regular points $R$
such that $\alpha(x)\subset H\cup S$ and $\omega(x)\subset H\cup S$ for every $x\in R$.
\end{theorem}

Since the intersection of a positively and a negatively sectional-hyperbolic set is both positively and negatively sectional-hyperbolic,
we obtain the following corollary.

\begin{corollary}
\label{c3}
The intersection of
a positively and a negatively sectional-hyperbolic set
is a disjoint union of
a (possibly empty) nonsingular hyperbolic set $H$,
a (possibly empty) finite set of singularities $S$
and a (possibly empty) set of regular points $R$
such that $\alpha(x)\subset H\cup S$ and $\omega(x)\subset H\cup S$ for every $x\in R$.
\end{corollary}

Our next result is an example of nontrivial transitive sets
which are positively and negatively sectional-hyperbolic (resp.)
whose intersection is the simplest possible, i.e., a single periodic orbit.

Denote by $\cl(\cdot)$ the closure operation.
We say that $H\subset M$ is a {\em homoclinic class} if there is a hyperbolic periodic point $x$ of saddle type
such that
$$
H=\cl(\{q\in W^{ws}(x)\cap W^{wu}(x): \dim(T_qW^{ws}(x)\cap T_qW^{wu}(x))=1\}).
$$
It follows from the {\em Birkhoff-Smale Theorem} that every homoclinic class is a transitive set with dense periodic orbits.

Given points $x,y\in M$, if for every $\epsilon>0$ there are sequences of points $\{x_i\}_{i=0}^n$
and times $\{t_i\}_{i=0}^{n-1}$ such that
$x_0=x$, $x_n=y$, $t_i\geq1$ and $d(X_{t_i}(x_i),x_{i+1})<\epsilon$ for every $0\leq i\leq n-1$,
then we say that $x$ {\em is in the chain stable set of $y$}.
If $x$ is in the chain stable set of $y$ and viceversa, then one says that
$x$ and $y$ are {\em chain related}.
If $x$ is chain related to itself, one says that $x$ is a {\em chain recurrent point}.
The set of chain recurrent points is the {\em chain recurrent set} denoted by $CR(X)$.
It is clear that the chain related relation is in equivalence on $CR(X)$.
By using this equivalence, one splits $CR(X)$ into equivalence classes
denominated {\em chain recurrent classes}.

A flow is {\em star} if it exhibits a neighborhood $\mathcal{U}$ (in the space of $C^1$ flows)
such that every periodic orbit or singularity of every flow in $\mathcal{U}$ is hyperbolic.

With these definitions we obtain the following result.

\begin{theorem}
\label{thB}
There is a star flow $X$ in the sphere $S^3$ whose chain recurrent set is the disjoint union
of two periodic orbits
$O_1$ (a sink), $O_2$ (a source);
two singularities $s_-$ (a source), $s_+$ (a saddle); and two homoclinic classes $H_-$, $H_+$ with the following properties:
\begin{itemize}
 \item $H_-$ is negatively (but not positively) sectional-hyperbolic;
 \item $H_+$ is positively (but not negatively) sectional-hyperbolic;
 \item $H_-\cap H_+$ is a periodic orbit.
\end{itemize}
\end{theorem}

Recall that
the {\em nonwandering set} of a flow $X$ is defined as the set of points
$x\in M$ such that for every neighborhood $U$ of $x$ and $T>0$ there is $t\geq T$ satisfying
$X_t(U)\cap U\neq\emptyset$.
Given a certain subset $O$ of the space of $C^1$ flows,
we say that a $C^1$ generic flow in $O$ satisfies another property (Q)
if there is a residual subset of flows $R$ of $O$
such that every flow in $R$ satisfying (P) also satisfies (Q).

\medskip

There are two current conjectures relating star flows and sectional-hyperbolicity. These are
based on previous results in the literature e.g. \cite{gaw}, \cite{mp}.

\begin{conjecture}[Zhu-Shy-Gan-Wen \cite{sgw},\cite{zgw}]
\label{zgw}
The chain recurrent set of {\em every} star flow is
the {\em disjoint} union of a positively
sectional-hyperbolic set and a negatively sectional-hyperbolic set.
\end{conjecture}

\begin{conjecture}[Arbieto \cite{am}]
The nonwandering set of a {\em $C^1$ generic} star flow is the disjoint union
of finitely many transitive sets which are positively
or negatively sectional-hyperbolic.
\end{conjecture}

However, the union $H_-\cup H_+$ of the homoclinic classes $H_-$ and $H_+$ in Theorem \ref{thB}
is a chain recurrent class of the corresponding flow $X$ (because $H_-\cap H_+\neq\emptyset$). Therefore,
Theorem \ref{thB} gives a counterexample for Conjecture \ref{zgw} in dimension
$3$. Similar counterexamples can be obtained in dimension $\geq3$.

\begin{corollary}
\label{c2}
There is a star flow in $S^3$ whose chain recurrent set is not
the disjoint union of a positively sectional-hyperbolic set
and a negatively sectional-hyperbolic set.
\end{corollary}

Another interesting feature regarding this counterexample
is the existence of a chain recurrent class without any nontrivial dominated splitting with respect
to the tangent flow. Moreover, every ergodic measure supported on this class is hyperbolic saddle.
These features are related to \cite{cs} or \cite{m}.
Notice also that the star flow in Corollary \ref{c2}
can be $C^1$ approximated by ones exhibiting the heteroclinic cycle
obtained by joinning the unstable manifold $W^u(\sigma_1)$ of $\sigma_1$ to the stable manifold $W^s(\sigma_2)$ of $\sigma_2$ in Figure \ref{fig1}.
Such a cycle was emphasized in the figure after the statement of
Lemma 3.3 in p.951 of \cite{zgw}.
This put in evidence the role of robust transitivity in the proof of such a lemma. 

\section*{Acknowledgments}

\noindent
This paper grew out of discussions between authors and participants of the
course {\em Topics in Dynamical Systems II} given at the Federal University of Rio de Janeiro, Brazil,
in the last half of 2014. The authors express their gratitute to these participants including professors J. Aponte, T. Catalan,
A.M. Lopez B. and H. Sanchez.
The results in this paper were announced in the
{\em I Workshop on Sectional-Anosov flows} which took place in September 22 of 2014 at the Federal University of Vi\c cosa-MG, Brasil.
The authors would like to thank professors E. Apaza and B. Mejia for the corresponding invitation.

\section{Proof of theorems \ref{th2} and \ref{th3}}

\noindent
First we prove Theorem \ref{th3}. For this we use the following technical definition.

\begin{defi}
A compact invariant set $\Lambda$ of a flow $X$ is {\em almost hyperbolic} if:
\begin{enumerate}
 \item Every singularity in $\Lambda$ is hyperbolic.
 \item There are continuous invariant subbundles $E^s, E^u$ of $T_\Lambda M$ such that
 $E^s$ is contracting, $E^u$ is expanding and
 $$
 T_{\Lambda^*}M=E^s\oplus E^X\oplus E^u.
 $$
\end{enumerate}
\end{defi}

Notice that this definition is symmetric with respect to the
reversing-flow operation.
Moreover, hyperbolic sets are almost hyperbolic but not conversely
by the example in Figure \ref{fig1}.
Likewise sectional-hyperbolic sets,
the almost hyperbolic sets satisfy

\begin{lemma}[Hyperbolic Lemma]
Every compact invariant subset without singularities of an almost
periodic set is hyperbolic.
\end{lemma}

More properties will be obtained from the
lemma below. We denote by $B(x,\delta)$ the open $\delta$-ball operation, $\delta>0$.
If $\sigma\in Sing(X)$ is hyperbolic, then we denote
by $W^s_\delta(\sigma)$ (resp. $W^u_\delta(\sigma)$) the connected component
of $B(\sigma,\delta)\cap W^s(\sigma)$ (resp. $B(\sigma,\delta)\cap W^u(\sigma)$)
containing $\sigma$.

\begin{lemma}
\label{thA'}
For every almost hyperbolic set $\Lambda$ of a flow $X$
there is $\delta>0$ such that
$\Lambda\cap B(\sigma,\delta)\subset W^u_\delta(\sigma)\cup W^u_\delta(\sigma)$ for every
$\sigma\in \Sing(X)\cap \Lambda$.
\end{lemma}

\begin{proof}
It suffices to prove that if $x_n\in \Lambda^*$ is a sequence converging to some singularity $\sigma\in \Lambda$,
then $x_n\in W^s(\sigma)\cup W^u(\sigma)$ for $n$ large enoch.

Let $T_\sigma M=F^s_\sigma\oplus F^u_\sigma$
be the hyperbolic splitting of $\sigma$.
By definition
$T_{x_n}M=E^s_{x_n}\oplus E_{x_n}^X\oplus E^u_{x_n}$ so
$$
\dim(E^s_{x_n})+\dim(E^u_{x_n})=\dim(M)-1,
\quad\quad\forall n.
$$
Passing to the limit we obtain
$$
\dim(E^s_{\sigma})+\dim(E^u_{\sigma})=\dim(M)-1.
$$
Since $E^s_\sigma$ and $E^u_\sigma$ are contracting and expanding respectively, we obtain
$E^s_\sigma\subset F^s_\sigma$ and $E^u_\sigma\subset F^u_\sigma$.

If $\dim(F^s_\sigma)>\dim(E^s_\sigma)+1$ we would have
$$
\dim(E^u)=\dim(M)-1-\dim(E^s_\sigma)>\dim(M)-\dim(F^s_\sigma)=\dim(F^u_\sigma),
$$
which is impossible.
Then $\dim(F^s_\sigma)\leq \dim(E^s_\sigma)+1$.
Analogously, $\dim(F^u_\sigma)\leq \dim(E^u_\sigma)+1$.
Therefore, $\dim(F^s_\sigma)=\dim(E^s_\sigma)$ or $\dim(E^s_\sigma)+1$.
Analogously $\dim(F^u_\sigma)=\dim(E^u_\sigma)$ or $\dim(E^u_\sigma)+1$.

But we cannot have $\dim(F^s_\sigma)=\dim(E^s_\sigma)+1$ and $\dim(F^u_\sigma)=\dim(E^u_\sigma)+1$
simultaneously because
$$
\dim(M)=\dim(F^s_\sigma)+\dim(F^u_\sigma)
=\dim(E^s_\sigma)+\dim(E^u_\sigma)+2=\dim(M)+1
$$
which is absurd.
All together imply
$$
\dim(E^s_\sigma)=\dim(F^s_\sigma)\quad \quad\mbox{ or } \quad\quad  \dim(E^u_\sigma)=\dim(F^u_\sigma).
$$

Suppose $\dim(E^s_\sigma)=\dim(F^s_\sigma)$. If $y\in \Lambda\cap(W^s(\sigma)\setminus\{\sigma\})$ is sufficiently close to
$\sigma$, then
$\dim(E^s_y)=\dim(E^s_\sigma)=\dim(F^s_\sigma)=\dim(T_yW^s(\sigma))$.

On the other hand, $E^s$ is contracting thus
$E^s_y\subset T_yW^s(\sigma)$.
From these remarks we obtain that if $\dim(E^s_\sigma)=\dim(F^s_\sigma)$, then
$E^s_y=T_yW^s(\sigma)$ for all $y\in \Lambda\cap( W^s(\sigma)\setminus\{\sigma\})$ close to $\sigma$.
Analogously if $\dim(E^u_\sigma)=\dim(F^u_\sigma)$,
then $E^u_y=T_yW^u(\sigma)$ for all $y\in \Lambda\cap(W^u(\sigma)\setminus\{\sigma\})$ close to $\sigma$.

Now suppose by contradiction that $x_n\not\in W^s(\sigma)\cup W^u(\sigma)$
for all $n$ (say).
Then, by flowing the orbit of $x_n$ nearby $\sigma$, as described in Figure \ref{fig2}, we obtain
two sequences $x_n^s,x^u_n$ in the orbit of $x_n$
such that $x^s_n\to y^s$ and $x^u_n\to y^u$ for some
$y^s\in W^s(\sigma)\setminus \{\sigma\}$ and $y^u\in W^u(\sigma)\setminus\{\sigma\}$
close to $\sigma$.

\begin{figure}[htv] 
\begin{center}
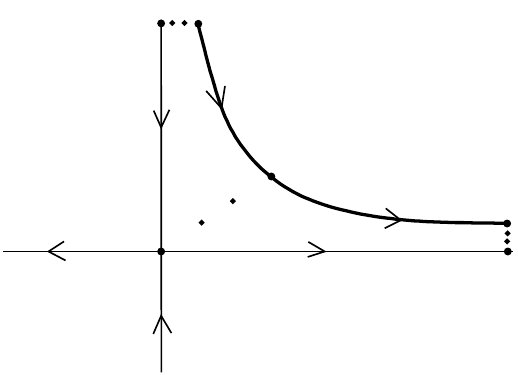 
\caption{Proof of Lemma \ref{thA'}}\label{fig2}
\end{center}
\end{figure}

If $\dim(E^s_\sigma)=\dim(F^s_\sigma)$ then $E^s_{y^s}=T_{y^s}W^s(\sigma)$
but also $E^X_{y^s}\subset T_{y^s}W^s(\sigma)$ since $W^s(\sigma)$ is an invariant manifold.
Therefore, $E^X_{y^s}\subset E^s_{y^s}$ and then $E^X_{y^s}=0$ since the sum
$T_{y^s}M=E^s_{y^s}\oplus E^X_{y^s}\oplus E^u_{y^s}$ is direct. This is a contradiction.
Analogously we obtain a contradiction if $\dim(E^u_\sigma)=\dim(F^u_\sigma)$
and the proof follows.
\end{proof}

Now we relate sectional and almost hyperbolicity.

\begin{lemma}
\label{thA}
Every compact invariant set which is both positively and negatively sectional-hyperbolic is almost hyperbolic.
\end{lemma}

\begin{proof}
Let $\Lambda$ be a compact invariant set which is
both positively and negatively sectional-hyperbolic.
Then, every singularity in $\Lambda$ is hyperbolic.
Moreover, there are positively and negatively sectional-hyperbolic splittings
$$
T_\Lambda M=E^s\oplus E^{se},
\quad\mbox{ and }\quad
T_\Lambda M=\hat{E}^s\oplus \hat{E}^{se},
$$
Taking $E^u=\hat{E}^s$ and $E^{sc}=\hat{E}^{se}$ we obtain an expanding and a sectional contracting subbundles of $T_\Lambda M$.
Since $E^s$ is contracting, we have $E^X\subset E^{se}$ by Lemma 3.2 in \cite{as}.
Similarly, $E^X\subset E^{sc}$ so
$$
E^X\subset E^{se}\cap E^{sc}.
$$

On the other hand, since $E^s$ is contracting and $E^u$ expanding,
the angle $\langle E^s, E^u\rangle$ is bounded away from zero.
Then, the dominating condition implies
$$
E^u\subset E^{se} \quad\quad\mbox{ and } \quad\quad E^s\subset E^{sc}.
$$
From this we have
$T_\Lambda M=E^{se}+ E^{sc}$
and so
$$
\dim(M)=\dim(E^{se})+\dim(E^{sc})-\dim(E^{se}\cap E^{sc}).
$$

At regular points we cannot have a vector outside $E^X$ contained in $E^{se}\cap E^{sc}$. Then,
$$
E^X=E^{se}\cap E^{sc}
\quad\quad\mbox{ and so }\quad\quad \dim(E^{se}\cap E^{sc})=1
$$
in $\Lambda^*$.
Replacing above we get
$$
\dim(M)=\dim(E^{se})+\dim(E^{sc})-1.
$$
But we also have
$\dim(M)=\dim(E^u)+\dim(E^{sc})$
so
$$
\dim(E^u)=\dim(E^{se})-1.
$$
Since $E^u$ is expanding, we have $E^u\cap E^X=\{0\}$ thus
$$
T_{\Lambda^*} M=E^s\oplus E^X\oplus E^u
$$
proving the result.
\end{proof}

\begin{proof}[Proof of Theorem \ref{th1}]
Let $\Lambda$ be the intersection of a positively and a negatively sectional-hyperbolic set of a flow $X$.
Then, it is both positively and negatively sectional-hyperbolic
and so almost hyperbolic by Lemma \ref{thA}.
From this we can select
$\delta>0$ as in Lemma \ref{thA'}. Clearly we can take $\delta$ such that
the balls $B(\sigma,\delta)$ are pairwise disjoint for $\sigma\in S$,
where $S=\Sing(X)\cap \Lambda$.

Define
$$
H=\displaystyle\bigcap_{(t,\sigma)\in\mathbb{R}\times S}X_t(\Lambda\setminus B(\sigma,\delta))
$$
and $R=\Lambda\setminus (H\cup S)$.

Clearly $S$ consists of finitely many singularities.
Moreover, $H$ is nonsingular hence hyperbolic by the Hyperbolic Lemma.
Now take $x\in R$.
Then, there is $(t,\sigma)\in\mathbb{R}\times S$ such that $X_t(x)\in B(\sigma,\delta)$.
By Lemma \ref{thA'} we obtain $X_t(x)\in W^s(\sigma)\cup W^u(\sigma)$ hence
$x\in W^s(\sigma)\cup W^u(\sigma)$.

If $x\in W^s(\sigma)$ we obtain $\omega(x)\subset H\cup S$.
If $X_r(x)\notin \cup_{\sigma\in S}B(\sigma,\delta)$ for all $r\leq 0$ then
$\alpha(x)\subset H$.
Otherwise, there is $(r,\rho)\in \mathbb{R}\times S$ such that
$X_r(x)\in B(\rho,\delta)$
and so $x\in W^u(\rho)$. All together yields $\alpha(x)\subset H\cup S$.
Similarly we have $\alpha(x)\subset H\cup S$ and
$\omega(x)\subset H\cup S$ if $x\in W^u(\sigma)$ and the result follows.
\end{proof}

To prove Theorem \ref{th2} we use the following lemma.
Recall that an invariant set is {\em nontrivial} if it does not reduces to a single orbit.

\begin{lemma}
\label{palilla}
Let $\Lambda$ be a nontrivial transitive positively sectional-hyperbolic set of a flow $X$.
If $\sigma\in\Sing(X)\cap \Lambda$, then the hyperbolic and the respective hyperbolic and positively sectional-hyperbolic splittings
$T_\sigma M=F^s_\sigma\oplus F^u_\sigma$ and $T_\sigma M=E^s_\sigma\oplus E^{se}_\sigma$ of $\sigma$
satisfy $\dim(E^{se}_\sigma\cap F^s_\sigma)=1$.
\end{lemma}

\begin{proof}
Clearly $E^s_\sigma\subset F^s_\sigma$.
Suppose for a while that $E^s_\sigma=F^s_\sigma$.
Then, $\dim(E^s_y)=\dim(T_yW^s(\sigma))$ for every $y\in \Lambda\cap W^s(\sigma)$ close to $\sigma$.
As clearly $E^s_y\subset T_yW^s(\sigma)$ for all such points $y$, we obtain
$E^s_y=T_yW^s(\sigma)$ for every $y\in \Lambda\cap W^s(\sigma)$ close to $\sigma$.
On the other hand, we also have that $E^X_y\subset T_y W^s(\sigma)$ for all such points $y$.
From this we conclude that $E^X_y\subset E^s_y$ for every point $y\in\Lambda\cap W^s(\sigma)$ close to $\sigma$.
Now we observe that since $\Lambda$ is transitive we obtain $E^X\subset E^{se}$.
Using again that $\Lambda$ is nontrivial transitive (see Figure \ref{fig2}) we obtain $y=y^s\in \Lambda^*\cap W^s(\sigma)$
close to $\sigma$. For such a point we obtain $0\neq E^X_y\subset E^s_y\cap E^{se}_y$ which is absurd.
Therefore, $E^s_\sigma\neq F^s_\sigma$.

Next we observe that $\dim(E^{se}_\sigma\cap F^s_\sigma)\leq1$ by sectional expansivity.
Suppose for a while that $\dim(E^{se}_\sigma\cap F^s_\sigma)=0$.
Clearly $E^s_\sigma\cap F^u_\sigma=0$ and so $F^u_\sigma\subset E^{se}_\sigma$
by domination.
From this we obtain $T_\sigma M=E^{se}_\sigma\oplus F^s_\sigma$
thus
$\dim(E^{se}_\sigma)+\dim(F^s_\sigma)=\dim(M)=\dim(F^s_\sigma)+\dim(F^u_\sigma)$ yielding
$\dim(E^{se}_\sigma)=\dim(F^u_\sigma)$ so $E^{se}_\sigma=F^u_\sigma$ thus $E^s_\sigma=F^s_\sigma$ which is absurd.
Therefore, $\dim(E^{se}_\sigma\cap F^s_\sigma)=1$ and we are done.
\end{proof}

\begin{proof}[Proof of Theorem \ref{th2}]
Let $\Lambda_+$ and $\Lambda_-$ be transitive sets of a flow $X$ such
that $\Lambda_+$ is positively sectional hyperbolic and $\Lambda_-$ is negatively sectional-hyperbolic.
If one of these sets reduces to a single orbit, then the intersection $\Lambda_-\cap \Lambda$ reduces to
that orbit and the result follows.

So, we can assume both $\Lambda_+$ and $\Lambda_-$ are nontrivial.
Let $T_{\Lambda_+}M=E^s\oplus E^{se}$ and $T_{\Lambda_-}M=\hat{E}^s\oplus \hat{E}^{se}$
be the positively and negatively sectional-hyperbolic splittings
of $\Lambda_+$ and $\Lambda_-$ respectively.
Denoting $E^u=\hat{E}^s$ and $E^{sc}=\hat{E}^{se}$ we obtain an expanding subbundle and a
sectionally contracting subbundle of $T_\Lambda M$.

Suppose for a while that there is $\sigma\in \Lambda_-\cap\Lambda_+\cap\Sing(X)$.
By Lemma \ref{palilla} applied to $X$, we have that $\sigma$ has a real negative eigenvalues $\lambda^s$
corresponding to the one-dimensional eigendirection $E^{se}_\sigma\cap F^s_\sigma$.
Similarly, applying the lemma to $-X$, we obtain a real positive eigenvalue $\lambda^u$
corresponding to the one-dimensional eigendirection $E^{sc}_\sigma\cap F^u_\sigma$.

Take unitary vectors $v^s\in E^{se}_\sigma\cap F^s_\sigma$ and $v^u\in E^{sc}_\sigma\cap F^u_\sigma$.
Since
$$
(E^{se}_\sigma\cap F^s_\sigma)\cap (E^{sc}_\sigma\cap F^u_\sigma)\subset F^s_\sigma\cap F^u_\sigma=0,
$$
we have that $v^s$ and $v^u$ are linearly independent. Then,
$\|v^s,v^u\|\neq0$.
Since $F^u_\sigma\subset E^{se}_\sigma$, we have $v^s,v^u\in E^{se}_\sigma$
so
$$
e^{\lambda^st}e^{\lambda^ut}\|v^s,v^u\|=\|DX_t(\sigma)v^s,DX_t(\sigma)v^u\|\to\infty \quad\mbox{ as }\quad t\to\infty
$$
by sectionally expansiveness. Then
$$
\lambda^s+\lambda^u>0.
$$
Similarly, since $F^s_\sigma\subset E^{sc}_\sigma$, we have $v^s,v^u\in E^{sc}_\sigma$
so
$$
e^{-\lambda^st}e^{-\lambda^ut}\|v^s,v^u\|=\|DX_{-t}(\sigma)v^s,DX_{-t}(\sigma)v^u\|\to\infty \quad\mbox{ as }\quad t\to\infty
$$
by sectionally expansiveness with respect to $-X$. Then,
$$
\lambda^s+\lambda^u<0
$$
which is absurd. We conclude that $\Lambda_-\cap\Lambda_+\cap\Sing(X)=\emptyset$.
Now we can apply the hyperbolic lemma for sectional-hyperbolic sets to obtain
that $\Lambda_-\cap\Lambda_+$ is hyperpolic. This finishes the proof.
\end{proof}

\section{Proof of Theorem \ref{thB}}

\noindent
Roughly speaking, the proof consists of glueing the so-called {\em singular horseshoe} \cite{lp}
with its time reversed counterpart.

\begin{figure}[h] 
\begin{center}
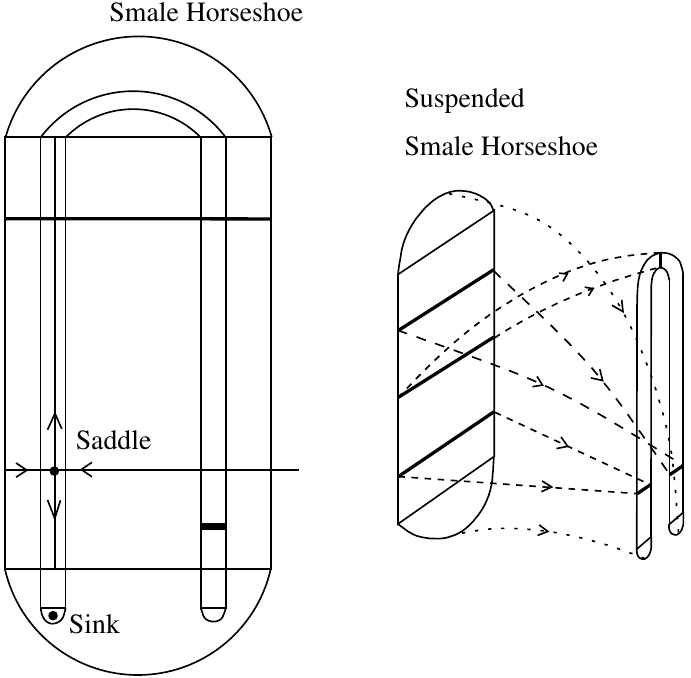 
\caption{}\label{fig3}
\end{center}
\end{figure}

We star with the standard {\em Smale horseshoe} which is the map in the 2-disk on the left of Figure \ref{fig3}.
It turns out that its nonwandering set consists of a sink and a hyperbolic homoclinic class
containing the saddle. Its suspension is the flow described in the right-hand picture of the figure.
It is a flow in the solid torus whose nonwandering set
is also a periodic sink $O_1$ together with a hyperbolic homoclinic class.


The next Figure \ref{fig4} describes a procedure of inserting singularities in the suspended Smale horseshoe.
We select an horizontal interval $I$ and a point $x$ in the square forming the horseshoe.

\begin{figure}[h] 
\begin{center}
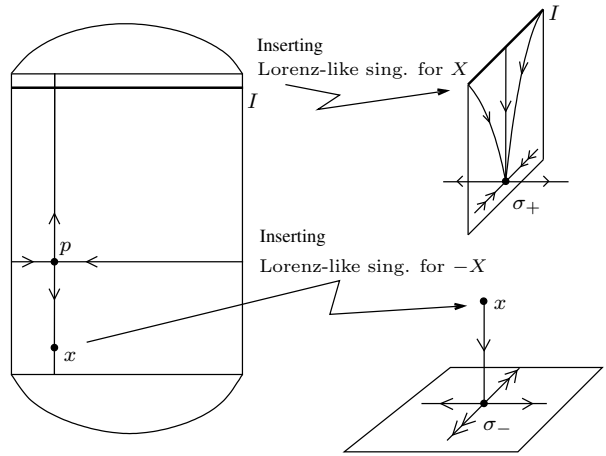 
\caption{Inserting singularities.}\label{fig4}
\end{center}
\end{figure}

The selection is done in order to place $I$ in the stable manifold of a Lorenz-like equilibrium $\sigma_+$,
and $x$ in the stable manifold of a Lorenz-like equilibrium for the reversed flow $\sigma_-$.
This construction requires to add two additional singularities,
a source $s_-$ to which the unstable branch of $\sigma_-$
not containig $x$ goes; and a saddle $s_+$ close to $\sigma_+$.
See Figure \ref{fig0-5}.

\begin{figure}[h] 
\begin{center}
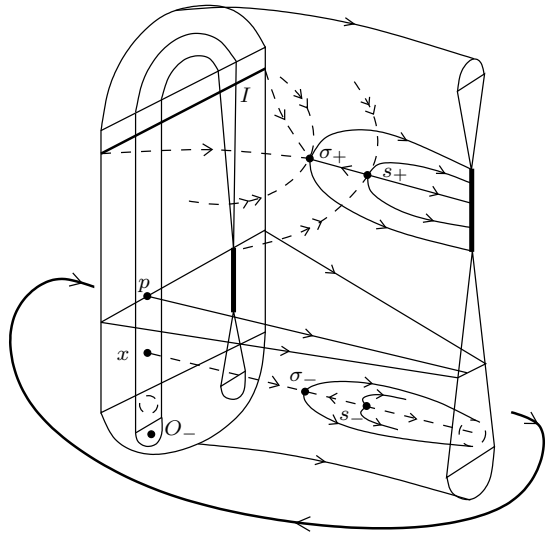 
\caption{Still inserting singularities.}\label{fig0-5}
\end{center}
\end{figure}

An accurate description of the aforementioned procedure is done
in \cite{bpv} and \cite{n}.

Next we observe that the resulting flow's return map presents a cut along $I$ and a blowup circle
derived from $x$.

We now proceed to deform the flow in order to obtain a deformation of the return map
by pushing up one branch of the circle, and pushing down the cusped region derived from the cutting as indicated in Figure \ref{fig5}.

\begin{figure}[h] 
\begin{center}
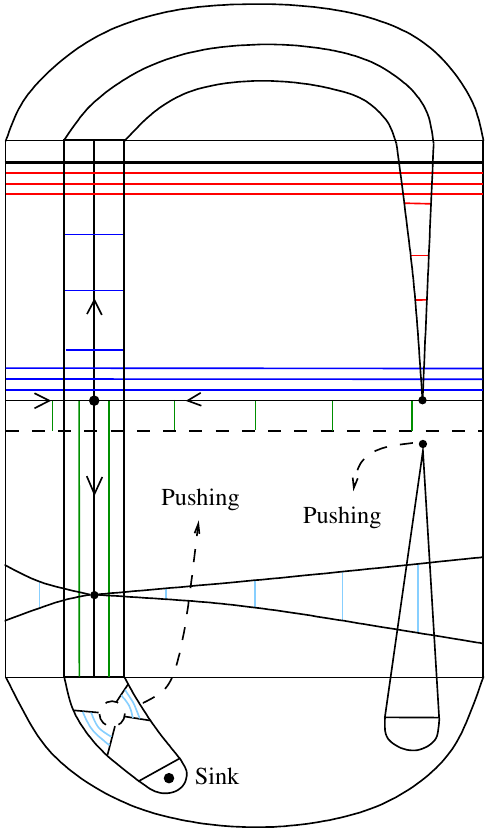 
\caption{Deforming.}\label{fig5}
\end{center}
\end{figure}

We keep doing this deformation (see Figure \ref{fig6}) up to arrive to the final flow whose
return map is described in Figure \ref{fig7}.

\begin{figure}[h] 
\begin{center}
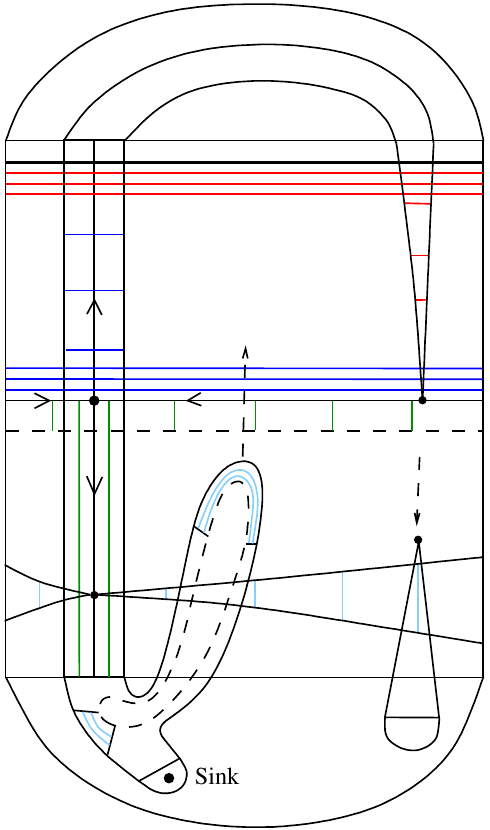 
\caption{Still deforming.}\label{fig6}
\end{center}
\end{figure}

\begin{figure}[h] 
\begin{center}
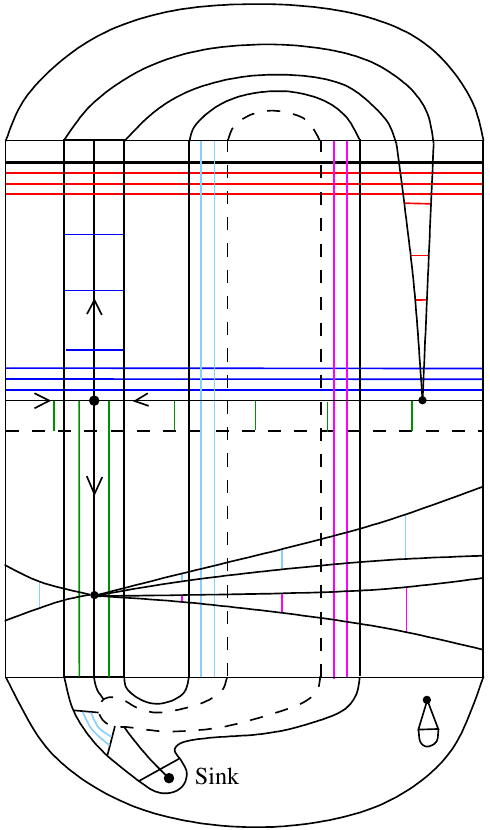 
\caption{Return map.}\label{fig7}
\end{center}
\end{figure}

This flow is defined in a solid torus, transversal to the boundary and pointing
inward there.

The final return map (denoted by $R$) is described with some detail in Figure \ref{fig8}.

We are in position to describe the homoclinic classes $H_-$ and $H_+$
in Theorem \ref{thB}.
They are precisely the maximal invariant set of $R$ in the upper and lower rectangles $Q_+$ and $Q_-$ forming the rectangle $Q$
in Figure \ref{fig8}. These maximal sets are located in the intersections
$A\cap B\cap A'\cap B'$ (for $H_-$) and $C\cap D\cap D\cap E\cap C'\cap E'\cap D'$ (for $H_+$).
A rough description of $H_-$ and $H_+$ is
that $H_+$ is the singular horseshoe in \cite{lp} and $H_-$ its time reversal.

\begin{figure}[h] 
\begin{center}
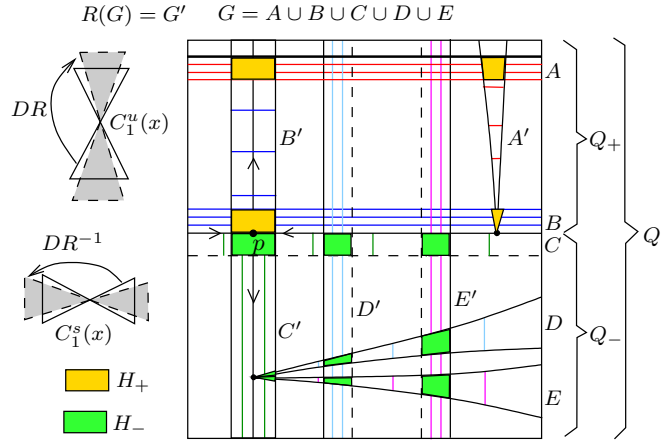 
\caption{Localizing $H-$ and $H_+$ in $Q$.}\label{fig8}
\end{center}
\end{figure}

The proof that $H_-$ and $H_+$ are nontrivial homoclinic classes is done as in \cite{b},
\cite{bmp}.
The analysis in \cite{blmp} or \cite{lp} shows that $H_+$ is a sectional-hyperbolic set for the (final) flow
and that $H_+$ is a sectional-hyperbolic set for the reversed flow.
We assume that the horizontal conefield
$x\in Q\mapsto C^s_1(x)$ and the vertical conefield
$x\in Q\mapsto C^u_1(x)$, where
$$
C^s_\alpha(x)=\left\{(a,b)\in\mathbb{R}^2:\frac{|b|}{|a|}\leq\alpha \right\}
\mbox{ and }
C^u_\alpha(x)=\left\{(a,b)\in\mathbb{R}^2: \frac{|a|}{|b|}\leq \alpha\right\}, \quad\forall \alpha>0,
$$
are contracting and expanding (respectively) for the return map $R$ in the sense that
there is $\rho>1$ with the following properties:

\begin{enumerate}
\item If $x\in R^{-1}(Q)\cap Q$ then
$$
DR(x)C^u_1(x)\subset C^u_{\frac{1}{2}}(R(x))
\mbox{ and }
\|DR(x)v^u\|\geq\rho\|v^u\|, \quad\forall v^u\in C^u_1(x).
$$
\item If $x\in R(Q)\cap Q$ then
$$
DR^{-1}(x)C^s_1(x)\subset C^s_{\frac{1}{2}}(R^{-1}(x))
\mbox{ and }
\|DR^{-1}(x)v^s\|\geq \rho\|v^s\|, \quad\forall v^s\in C^s_1(x).
$$
\end{enumerate}
(See Figure \ref{fig8}.)
Since such conefields do not allow the existence of nonhyperbolic periodic points,
and are preserved by small perturbations,
we obtain that the final flow is star in its solid torus domain.

Next we observe that
$H_-$ is not hyperbolic, since it contains the singularity $\sigma_-$ and, analogously,
$H_+$ is not hyperbolic for it contains $\sigma_+$.
Since every homoclinic class is transitive, we conclude from
Theorem \ref{th1} that $H_-$ is not positively sectional-hyperbolic and $H_+$ is not negatively sectional-hyperbolic.

To complete the proof we extend the final flow from its solid torus domain
to the whole $S^3$. This is done by glueing it with another solid torus
whose core is a periodic source $O_2$.
This completes the proof.
\qed

\end{document}

%% file: vicosa01.pdf_t
\begin{picture}(0,0)%
\includegraphics{vicosa01.pdf}%
\end{picture}%
\setlength{\unitlength}{4144sp}%
\begingroup\makeatletter\ifx\SetFigFont\undefined%
\gdef\SetFigFont#1#2#3#4#5{%
  \reset@font\fontsize{#1}{#2pt}%
  \fontfamily{#3}\fontseries{#4}\fontshape{#5}%
  \selectfont}%
\fi\endgroup%
\begin{picture}(3894,1728)(-11,-867)
\put(3316,-306){\makebox(0,0)[lb]{\smash{{\SetFigFont{9}{10.8}{\rmdefault}{\mddefault}{\updefault}{\color[rgb]{0,0,0}$\sigma_2$}%
}}}}
\put(3342,353){\makebox(0,0)[lb]{\smash{{\SetFigFont{9}{10.8}{\rmdefault}{\mddefault}{\updefault}{\color[rgb]{0,0,0}$\gamma^s$}%
}}}}
\put(2481,-338){\makebox(0,0)[lb]{\smash{{\SetFigFont{9}{10.8}{\rmdefault}{\mddefault}{\updefault}$W^u(\sigma_2)$}}}}
\put(2076,-261){\makebox(0,0)[lb]{\smash{{\SetFigFont{9}{10.8}{\rmdefault}{\mddefault}{\updefault}{\color[rgb]{0,0,0}$x$}%
}}}}
\put(594,-271){\makebox(0,0)[lb]{\smash{{\SetFigFont{9}{10.8}{\rmdefault}{\mddefault}{\updefault}{\color[rgb]{0,0,0}$\sigma_1$}%
}}}}
\put(1325,-303){\makebox(0,0)[lb]{\smash{{\SetFigFont{9}{10.8}{\rmdefault}{\mddefault}{\updefault}$W^s(\sigma_1)$}}}}
\put(3242,716){\makebox(0,0)[lb]{\smash{{\SetFigFont{9}{10.8}{\rmdefault}{\mddefault}{\updefault}$W^s(\sigma_2)$}}}}
\put(524,734){\makebox(0,0)[lb]{\smash{{\SetFigFont{9}{10.8}{\rmdefault}{\mddefault}{\updefault}{\color[rgb]{0,0,0}$W^u(\sigma_1)$}%
}}}}
\put(415,355){\makebox(0,0)[lb]{\smash{{\SetFigFont{9}{10.8}{\rmdefault}{\mddefault}{\updefault}{\color[rgb]{0,0,0}$\lambda^u$}%
}}}}
\put( 61,-127){\makebox(0,0)[lb]{\smash{{\SetFigFont{9}{10.8}{\rmdefault}{\mddefault}{\updefault}{\color[rgb]{0,0,0}$\lambda^s$}%
}}}}
\put(1443,681){\makebox(0,0)[lb]{\smash{{\SetFigFont{9}{10.8}{\rmdefault}{\mddefault}{\updefault}{\color[rgb]{0,0,0}$\Lambda=\{\sigma_1 , \sigma_2\} \cup O(x)$}%
}}}}
\put(2248,-804){\makebox(0,0)[lb]{\smash{{\SetFigFont{9}{10.8}{\rmdefault}{\mddefault}{\updefault}{\color[rgb]{0,0,0}Lorenz-like for $-X: \ \gamma^u < - \gamma^s$ }%
}}}}
\put( 15,-817){\makebox(0,0)[lb]{\smash{{\SetFigFont{9}{10.8}{\rmdefault}{\mddefault}{\updefault}{\color[rgb]{0,0,0}Lorenz-like for $X:\ \lambda^u > -\lambda^s$ }%
}}}}
\put(3786,-201){\makebox(0,0)[lb]{\smash{{\SetFigFont{9}{10.8}{\rmdefault}{\mddefault}{\updefault}{\color[rgb]{0,0,0}$\gamma^u$}%
}}}}
\end{picture}%

%% file: vicosa2-5-11.pdf_t
\begin{picture}(0,0)%
\includegraphics{vicosa2-5-11.pdf}%
\end{picture}%
\setlength{\unitlength}{4144sp}%
\begingroup\makeatletter\ifx\SetFigFont\undefined%
\gdef\SetFigFont#1#2#3#4#5{%
  \reset@font\fontsize{#1}{#2pt}%
  \fontfamily{#3}\fontseries{#4}\fontshape{#5}%
  \selectfont}%
\fi\endgroup%
\begin{picture}(2360,1701)(-4,-860)
\put(500,-423){\makebox(0,0)[lb]{\smash{{\SetFigFont{8}{9.6}{\rmdefault}{\mddefault}{\updefault}{\color[rgb]{0,0,0}$\sigma$}%
}}}}
\put(1281, 98){\makebox(0,0)[lb]{\smash{{\SetFigFont{8}{9.6}{\rmdefault}{\mddefault}{\updefault}{\color[rgb]{0,0,0}$x_n$}%
}}}}
\put(576,722){\makebox(0,0)[lb]{\smash{{\SetFigFont{8}{9.6}{\rmdefault}{\mddefault}{\updefault}{\color[rgb]{0,0,0}$y^s$}%
}}}}
\put(2263,-440){\makebox(0,0)[lb]{\smash{{\SetFigFont{8}{9.6}{\rmdefault}{\mddefault}{\updefault}{\color[rgb]{0,0,0}$y^u$}%
}}}}
\put(2235,-85){\makebox(0,0)[lb]{\smash{{\SetFigFont{8}{9.6}{\rmdefault}{\mddefault}{\updefault}{\color[rgb]{0,0,0}$x^u_n$}%
}}}}
\put(970,687){\makebox(0,0)[lb]{\smash{{\SetFigFont{8}{9.6}{\rmdefault}{\mddefault}{\updefault}{\color[rgb]{0,0,0}$x^s_n$}%
}}}}
\end{picture}%

%% file: vicosa110.pdf_t
\begin{picture}(0,0)%
\includegraphics{vicosa110.pdf}%
\end{picture}%
\setlength{\unitlength}{4144sp}%
\begingroup\makeatletter\ifx\SetFigFont\undefined%
\gdef\SetFigFont#1#2#3#4#5{%
  \reset@font\fontsize{#1}{#2pt}%
  \fontfamily{#3}\fontseries{#4}\fontshape{#5}%
  \selectfont}%
\fi\endgroup%
\begin{picture}(3145,3095)(11,-2267)
\end{picture}%

%% file: vicosa21.pdf_t
\begin{picture}(0,0)%
\includegraphics{vicosa21.pdf}%
\end{picture}%
\setlength{\unitlength}{4144sp}%
\begingroup\makeatletter\ifx\SetFigFont\undefined%
\gdef\SetFigFont#1#2#3#4#5{%
  \reset@font\fontsize{#1}{#2pt}%
  \fontfamily{#3}\fontseries{#4}\fontshape{#5}%
  \selectfont}%
\fi\endgroup%
\begin{picture}(3569,2697)(11,-1854)
\put(322,-640){\makebox(0,0)[lb]{\smash{{\SetFigFont{8}{9.6}{\rmdefault}{\mddefault}{\updefault}{\color[rgb]{0,0,0}$p$}%
}}}}
\put(3247,743){\makebox(0,0)[lb]{\smash{{\SetFigFont{8}{9.6}{\rmdefault}{\mddefault}{\updefault}{\color[rgb]{0,0,0}$I$}%
}}}}
\put(1517,-765){\makebox(0,0)[lb]{\smash{{\SetFigFont{7}{8.4}{\rmdefault}{\mddefault}{\updefault}{\color[rgb]{0,0,0}Lorenz-like sing. for $-X$}%
}}}}
\put(1444,215){\makebox(0,0)[lb]{\smash{{\SetFigFont{8}{9.6}{\rmdefault}{\mddefault}{\updefault}{\color[rgb]{0,0,0}$I$}%
}}}}
\put(1506,413){\makebox(0,0)[lb]{\smash{{\SetFigFont{7}{8.4}{\rmdefault}{\mddefault}{\updefault}{\color[rgb]{0,0,0}Lorenz-like sing. for $X$}%
}}}}
\put(3031,-388){\makebox(0,0)[lb]{\smash{{\SetFigFont{8}{9.6}{\rmdefault}{\mddefault}{\updefault}{\color[rgb]{0,0,0}$\sigma_{+}$}%
}}}}
\put(2920,-985){\makebox(0,0)[lb]{\smash{{\SetFigFont{8}{9.6}{\rmdefault}{\mddefault}{\updefault}{\color[rgb]{0,0,0}$x$}%
}}}}
\put(2856,-1712){\makebox(0,0)[lb]{\smash{{\SetFigFont{8}{9.6}{\rmdefault}{\mddefault}{\updefault}{\color[rgb]{0,0,0}$\sigma_{-}$}%
}}}}
\put(347,-1301){\makebox(0,0)[lb]{\smash{{\SetFigFont{8}{9.6}{\rmdefault}{\mddefault}{\updefault}{\color[rgb]{0,0,0}$x$}%
}}}}
\end{picture}%

%% file: vicosa1-21.pdf_t
\begin{picture}(0,0)%
\includegraphics{vicosa1-21.pdf}%
\end{picture}%
\setlength{\unitlength}{4144sp}%
\begingroup\makeatletter\ifx\SetFigFont\undefined%
\gdef\SetFigFont#1#2#3#4#5{%
  \reset@font\fontsize{#1}{#2pt}%
  \fontfamily{#3}\fontseries{#4}\fontshape{#5}%
  \selectfont}%
\fi\endgroup%
\begin{picture}(3236,3173)(44,-2362)
\put(703,-1325){\makebox(0,0)[lb]{\smash{{\SetFigFont{8}{9.6}{\rmdefault}{\mddefault}{\updefault}{\color[rgb]{0,0,0}$x$}%
}}}}
\put(834,-893){\makebox(0,0)[lb]{\smash{{\SetFigFont{8}{9.6}{\rmdefault}{\mddefault}{\updefault}{\color[rgb]{0,0,0}$p$}%
}}}}
\put(990,-1776){\makebox(0,0)[lb]{\smash{{\SetFigFont{8}{9.6}{\rmdefault}{\mddefault}{\updefault}{\color[rgb]{0,0,0}$O_{-}$}%
}}}}
\put(1737,-1452){\makebox(0,0)[lb]{\smash{{\SetFigFont{8}{9.6}{\rmdefault}{\mddefault}{\updefault}{\color[rgb]{0,0,0}$\sigma_{-}$}%
}}}}
\put(2033,-1685){\makebox(0,0)[lb]{\smash{{\SetFigFont{8}{9.6}{\rmdefault}{\mddefault}{\updefault}{\color[rgb]{0,0,0}$s_{-}$}%
}}}}
\put(1915,-118){\makebox(0,0)[lb]{\smash{{\SetFigFont{8}{9.6}{\rmdefault}{\mddefault}{\updefault}{\color[rgb]{0,0,0}$\sigma_{+}$}%
}}}}
\put(2292,-227){\makebox(0,0)[lb]{\smash{{\SetFigFont{8}{9.6}{\rmdefault}{\mddefault}{\updefault}{\color[rgb]{0,0,0}$s_{+}$}%
}}}}
\put(1439,219){\makebox(0,0)[lb]{\smash{{\SetFigFont{8}{9.6}{\rmdefault}{\mddefault}{\updefault}{\color[rgb]{0,0,0}$I$}%
}}}}
\end{picture}%

%% file: vicosa31.pdf_t
\begin{picture}(0,0)%
\includegraphics{vicosa31.pdf}%
\end{picture}%
\setlength{\unitlength}{4144sp}%
\begingroup\makeatletter\ifx\SetFigFont\undefined%
\gdef\SetFigFont#1#2#3#4#5{%
  \reset@font\fontsize{#1}{#2pt}%
  \fontfamily{#3}\fontseries{#4}\fontshape{#5}%
  \selectfont}%
\fi\endgroup%
\begin{picture}(2226,3787)(15,-2947)
\end{picture}%

%% file: vicosa41.pdf_t
\begin{picture}(0,0)%
\includegraphics{vicosa41.pdf}%
\end{picture}%
\setlength{\unitlength}{4144sp}%
\begingroup\makeatletter\ifx\SetFigFont\undefined%
\gdef\SetFigFont#1#2#3#4#5{%
  \reset@font\fontsize{#1}{#2pt}%
  \fontfamily{#3}\fontseries{#4}\fontshape{#5}%
  \selectfont}%
\fi\endgroup%
\begin{picture}(2226,3787)(15,-2947)
\end{picture}%

%% file: vicosa8-51.pdf_t
\begin{picture}(0,0)%
\includegraphics{vicosa8-51.pdf}%
\end{picture}%
\setlength{\unitlength}{4144sp}%
\begingroup\makeatletter\ifx\SetFigFont\undefined%
\gdef\SetFigFont#1#2#3#4#5{%
  \reset@font\fontsize{#1}{#2pt}%
  \fontfamily{#3}\fontseries{#4}\fontshape{#5}%
  \selectfont}%
\fi\endgroup%
\begin{picture}(2226,3787)(15,-2947)
\end{picture}%

%% file: vicosa61.pdf_t
\begin{picture}(0,0)%
\includegraphics{vicosa61.pdf}%
\end{picture}%
\setlength{\unitlength}{4144sp}%
\begingroup\makeatletter\ifx\SetFigFont\undefined%
\gdef\SetFigFont#1#2#3#4#5{%
  \reset@font\fontsize{#1}{#2pt}%
  \fontfamily{#3}\fontseries{#4}\fontshape{#5}%
  \selectfont}%
\fi\endgroup%
\begin{picture}(3812,2646)(9,-1809)
\put(286,-1203){\makebox(0,0)[lb]{\smash{{\SetFigFont{8}{9.6}{\rmdefault}{\mddefault}{\updefault}{\color[rgb]{0,0,0}$C_1^s(x)$}%
}}}}
\put(3806,-616){\makebox(0,0)[lb]{\smash{{\SetFigFont{8}{9.6}{\rmdefault}{\mddefault}{\updefault}{\color[rgb]{0,0,0}$Q$}%
}}}}
\put(3217,368){\makebox(0,0)[lb]{\smash{{\SetFigFont{8}{9.6}{\rmdefault}{\mddefault}{\updefault}{\color[rgb]{0,0,0}$A$}%
}}}}
\put(3215,-1588){\makebox(0,0)[lb]{\smash{{\SetFigFont{8}{9.6}{\rmdefault}{\mddefault}{\updefault}{\color[rgb]{0,0,0}$E$}%
}}}}
\put(3215,-1140){\makebox(0,0)[lb]{\smash{{\SetFigFont{8}{9.6}{\rmdefault}{\mddefault}{\updefault}{\color[rgb]{0,0,0}$D$}%
}}}}
\put(3217,-542){\makebox(0,0)[lb]{\smash{{\SetFigFont{8}{9.6}{\rmdefault}{\mddefault}{\updefault}{\color[rgb]{0,0,0}$B$}%
}}}}
\put(3215,-676){\makebox(0,0)[lb]{\smash{{\SetFigFont{8}{9.6}{\rmdefault}{\mddefault}{\updefault}{\color[rgb]{0,0,0}$C$}%
}}}}
\put(3484,-41){\makebox(0,0)[lb]{\smash{{\SetFigFont{8}{9.6}{\rmdefault}{\mddefault}{\updefault}{\color[rgb]{0,0,0}$Q_{+}$}%
}}}}
\put(3485,-1207){\makebox(0,0)[lb]{\smash{{\SetFigFont{8}{9.6}{\rmdefault}{\mddefault}{\updefault}{\color[rgb]{0,0,0}$Q_{-}$}%
}}}}
\put(613, 62){\makebox(0,0)[lb]{\smash{{\SetFigFont{8}{9.6}{\rmdefault}{\mddefault}{\updefault}{\color[rgb]{0,0,0}$C_1^u(x)$}%
}}}}
\put( 24,139){\makebox(0,0)[lb]{\smash{{\SetFigFont{8}{9.6}{\rmdefault}{\mddefault}{\updefault}{\color[rgb]{0,0,0}$DR$}%
}}}}
\put(2990,-32){\makebox(0,0)[lb]{\smash{{\SetFigFont{8}{9.6}{\rmdefault}{\mddefault}{\updefault}{\color[rgb]{0,0,0}$A'$}%
}}}}
\put(1627,-32){\makebox(0,0)[lb]{\smash{{\SetFigFont{8}{9.6}{\rmdefault}{\mddefault}{\updefault}{\color[rgb]{0,0,0}$B'$}%
}}}}
\put(2088,-1058){\makebox(0,0)[lb]{\smash{{\SetFigFont{8}{9.6}{\rmdefault}{\mddefault}{\updefault}{\color[rgb]{0,0,0}$D'$}%
}}}}
\put(2671,-999){\makebox(0,0)[lb]{\smash{{\SetFigFont{8}{9.6}{\rmdefault}{\mddefault}{\updefault}{\color[rgb]{0,0,0}$E'$}%
}}}}
\put(1621,-1187){\makebox(0,0)[lb]{\smash{{\SetFigFont{8}{9.6}{\rmdefault}{\mddefault}{\updefault}{\color[rgb]{0,0,0}$C'$}%
}}}}
\put(442,705){\makebox(0,0)[lb]{\smash{{\SetFigFont{8}{9.6}{\rmdefault}{\mddefault}{\updefault}{\color[rgb]{0,0,0}$R(G)=G'$}%
}}}}
\put(1262,714){\makebox(0,0)[lb]{\smash{{\SetFigFont{8}{9.6}{\rmdefault}{\mddefault}{\updefault}{\color[rgb]{0,0,0}$G = A \cup B \cup C \cup D \cup E$}%
}}}}
\put(209,-664){\makebox(0,0)[lb]{\smash{{\SetFigFont{8}{9.6}{\rmdefault}{\mddefault}{\updefault}{\color[rgb]{0,0,0}$DR^{-1}$}%
}}}}
\put(632,-1742){\makebox(0,0)[lb]{\smash{{\SetFigFont{8}{9.6}{\rmdefault}{\mddefault}{\updefault}{\color[rgb]{0,0,0}$H_{-}$}%
}}}}
\put(654,-1484){\makebox(0,0)[lb]{\smash{{\SetFigFont{8}{9.6}{\rmdefault}{\mddefault}{\updefault}{\color[rgb]{0,0,0}$H_{+}$}%
}}}}
\put(1469,-664){\makebox(0,0)[lb]{\smash{{\SetFigFont{9}{10.8}{\rmdefault}{\mddefault}{\updefault}{\color[rgb]{0,0,0}$p$}%
}}}}
\end{picture}%

%% file: sectional-hyperbolicity-2-inner-product11.bbl
\begin{thebibliography}{10}




\bibitem{abs}
Afraimovic, V.S.,  Bykov, V.V.;, Shilnikov, L.P.,
The origin and structure of the Lorenz attractor,
{\em Dokl. Akad. Nauk SSSR} 234 (1977), no. 2, 336--339. 





\bibitem{ap}
Ara\'ujo, V., Pacifico, M.J.,
{\em Three-dimensional flows.}
With a foreword by Marcelo Viana. Ergebnisse der Mathematik und ihrer Grenzgebiete. 3. Folge. A Series of Modern Surveys in
Mathematics [Results in Mathematics and Related Areas.
3rd Series. A Series of Modern Surveys in Mathematics], 53. Springer, Heidelberg, 2010.




\bibitem{as}
Ara\'ujo, V., Salgado, L.,
Infinitesimal Lyapunov functions for singular flows,
{\em Math. Z.} 275 (2013), no. 3-4, 863--897.



\bibitem{am}
Arbieto, A., Morales, C.A.,
A dichotomy for higher-dimensional flows,
{\em Proc. Amer. Math. Soc.} 141 (2013), no. 8, 2817--2827.



\bibitem{blmp}
Bamon, R., Labarca, R., Ma\~n\'e, R., Pacifico, M.J.,
The explosion of singular cycles,
{\em Inst. Hautes \'Etudes Sci. Publ. Math. No.} 78 (1993), 207--232. 





\bibitem{b}
Bautista, S.,
The geometric Lorenz attractor is a homoclinic class,
{\em Bol. Mat. (N.S.)} 11 (2004), no. 1, 69--78. 






\bibitem{bmp}
Bautista, S., Morales, C.A., Pacifico, M.J.,
On the intersection of homoclinic classes on singular-hyperbolic sets,
{\em Discrete Contin. Dyn. Syst.} 19 (2007), no. 4, 761--775. 








\bibitem{bpv}
Bonatti, C., Pumari\~no, A., Viana, M.,
Lorenz attractors with arbitrary expanding dimension,
{\em C. R. Acad. Sci. Paris S\'er. I Math.} 325 (1997), no. 8, 883--888. 






\bibitem{cs}
Carrasco-Olivera, D., San Mart\'{i}n, B.,
Robust attractors without dominated splitting on manifolds with boundary,
{\em Discrete Contin. Dyn. Syst.} 34 (2014), no. 11, 4555--4563. 






\bibitem{dgw} 
Diminnie, Ch., R., G\"ahler, S., White, A.,
$2$-inner product spaces,
Collection of articles dedicated to Stanislaw Golpolhkab on his 70th birthday, II.
{\em Demonstratio Math.} 6 (1973), 525--536.



\bibitem{gaw}
Gan, S., Wen, L.,
Nonsingular star flows satisfy Axiom A and the no-cycle condition,
{\em Invent. Math.} 164 (2006), no. 2, 279--315. 




\bibitem{g}
G\"ahler, S.,
Lineare 2-normierte R\"aume,
{\em Math. Nachr.} 28 (1964), 1--43. 






\bibitem{gw}
Guckenheimer, J.; Williams, R., Structural stability of Lorenz attractors,
{\em Publ. Math. IHES} 50 (1979), 59--72.










\bibitem{hps}
Hirsch, M., Pugh, C., Shub, M.,
{\em Invariant manifolds},
Lec. Not. in Math. 583 (1977), Springer-Verlag.





\bibitem{ka} 
Kawaguchi, A., 
{\em On areal spaces, I. Metric tensors in $n$-dimensional spaces based on the notion of two-dimensional area},
Tensor N.S.  1 (1950), 14--45.







\bibitem{lp}
Labarca, R., Pacifico, M.J.,
Stability of singularity horseshoes,
{\em Topology} 25 (1986), no. 3, 337--352. 








\bibitem{memo}
Metzger, R., Morales, C.A.,
Sectional-hyperbolic systems,
{\em Ergodic Theory Dynam. Systems} 28 (2008), no. 5, 1587--1597.


\bibitem{m}
Morales, C.A.,
Characterizing singular-hyperbolicity,
Preprint (2014), to appear.





\bibitem{mp}
Morales, C.A., Pacifico, M.J.,
A dichotomy for three-dimensional vector fields,
{\em Ergodic Theory Dynam. Systems} 23 (2003), no. 5, 1575--1600.

\bibitem{mp2}
Morales, C.A., Pacifico, M.J.
Sufficient conditions for robustness of attractors,
{\em Pacific J. Math.} 216 (2004), no. 2, 327--342. 


\bibitem{mp3}
Morales, C.A., Pacifico, M.J.,
A spectral decomposition for singular-hyperbolic sets,
{\em Pacific J. Math.} 229 (2007), no. 1, 223--232. 





\bibitem{mv}
Morales, C.A., Vilches, M.,
On 2-Riemannian manifolds,
{\em SUT J. Math.} 46 (2010), no. 1, 119--153. 








\bibitem{n}
Newhouse, S.,
On simple arcs between structurally stable flows.
Dynamical systems—Warwick 1974 (Proc. Sympos. Appl. Topology and Dynamical Systems, Univ. Warwick, Coventry, 1973/1974;
presented to E. C. Zeeman on his fiftieth birthday), pp. 209--233. {\em Lecture Notes in Math.}, Vol. 468, Springer, Berlin, 1975. 












\bibitem{sgw}
Shi, Y., Gan, S., Wen, L.,
On the singular hyperbolicity of star flows,
Preprint ArXiv http://de.arxiv.org/abs/1310.8118v1.



\bibitem{zgw}
Zhu, S., Gan, S., Wen, L.,
Indices of singularities of robustly transitive sets,
{\em Discrete Contin. Dyn. Syst.} 21 (2008), no. 3, 945--957. 


















\end{thebibliography}
